\documentclass[12pt,reqno]{amsart}
\usepackage{a4wide,xcolor,eucal,enumerate,mathrsfs,graphics,caption,subfig,booktabs}
\usepackage{pdfsync}
\usepackage{amsmath,amssymb,epsfig,amsthm}
\usepackage[foot]{amsaddr}
\usepackage[latin1]{inputenc}
\usepackage{dsfont}

\usepackage[english]{babel}
\usepackage{amsmath,amsfonts,amsthm,amssymb,dsfont,url}
\usepackage[babel]{csquotes}

\usepackage[style=numeric,backend=bibtex,doi=false,isbn=false,url=false]{biblatex}

\renewbibmacro{in:}{%
  \ifentrytype{article}{}{\printtext{\bibstring{in}\intitlepunct}}}
  \DeclareFieldFormat[article]{pages}{#1}%
  \DeclareFieldFormat[article]{title}{#1}%
  \DeclareFieldFormat[book]{title}{#1}%
  \DeclareFieldFormat[incollection]{title}{#1}%
  \DeclareFieldFormat[thesis]{title}{#1}%

\bibliography{shapeconstrained-estimation}

\theoremstyle{definition}
\newtheorem{de}{Definition}[section]

\theoremstyle{plain}
\newtheorem{theo}[de]{Theorem}
\newtheorem{lemma}[de]{Lemma}

\newtheorem{cor}[de]{Corollary}
\theoremstyle{remark}

\newcommand{\R}{\mathbb{R}}

\newcommand{\p}{\mathbb{P}}

\newcommand{\E}{\mathbb{E}}

\newcommand{\1}{\mathds{1}}

\def\argmax{\mathop{\text{argmax}}}

\title[]{A central limit theorem for the Hellinger loss of Grenander type estimators}
\author[Hendrik P. Lopuha\"a and Eni Musta]{Hendrik P. Lopuha\"a \and Eni Musta \\ Delft University of Technology}
\address{Delft Institute of Applied Mathematics, Mekelweg 4, 2628 CD Delft, The Netherlands}
\email{H.P.Lopuhaa@tudelft.nl, E.Musta@tudelft.nl}
\keywords{Hellinger distance, isotonic estimation, central limit theorem, Grenander estimator.}

\begin{document}
\maketitle
\begin{abstract}
We consider Grenander type estimators for a monotone function $\lambda:[0,1]\to\R$, obtained
as the slope of a concave (convex) estimate of the primitive of $\lambda$.
Our main result is a central limit theorem for the Hellinger loss,
which applies to statistical models that satisfy the setup in~\cite{durot2007}.
This includes estimation of a monotone density, for which the limiting variance of the Hellinger loss turns out to be independent of $\lambda$.

\end{abstract}

\section{Introduction}
\label{sec:introduction}
One of the problems in shape constrained nonparametric statistics is to estimate a real valued function under monotonicity constraints.
Early references for this type of problem can be found in~\cite{grenander1956},
\cite{brunk1958}, and~\cite{marshallproschan1965},
concerning the estimation of a probability density, a regression function, and a failure rate under monotonicity constraints.
The asymptotic distribution of these type of estimators was first obtained in~\cite{prakasarao1969,prakasarao1970}
and reproved in~\cite{groeneboom1985}, who introduced a more accessible approach on based inverses.
The latter approach initiated a stream of research on isotonic estimators, e.g.,
see~\cite{groeneboom-wellner1992,huang-zhang1994,huang-wellner1995,LopuhaaNane2013}.
Typically, the pointwise asymptotic behavior of isotonice estimators is characterized by a cube-root $n$ rate of convergence
and a non-normal limit distribution.

The situation is different for global distances.
In~\cite{groeneboom1985}, a central limit theorem was obtained for the $L_1$-error of the Grenander estimator of a monotone density
(see also~\cite{groeneboom-hooghiemstra-lopuhaa1999}) and a similar result was established in~\cite{durot2002} for the regression context.
Extensions to general $L_p$-errors can be found in~\cite{kulikov-lopuhaa2005} and in~\cite{durot2007}, where the latter provides a unified approach
that applies to a variety of statistical models.
For the same general setup, an extremal limit theorem for the supremum distance has been obtained in~\cite{DurotKulikovLopuhaa2012}.

Another widely used global measure of departure from the true parameter of interest is the Hellinger distance.
It is a convenient metric in maximum likelihood problems, which goes back to~\cite{lecam1970, lecam1973},
and it has nice connections with Bernstein norms and empirical process theory methods to obtain rates of convergence,
due fundamentally to~\cite{birgemassart1993}, \cite{wongshen1995}, and others, see Section~3.4 of~\cite{VW96} or Chapter 4 in~\cite{vdGeer2000} for a more detailed overview.
Consistency in Hellinger distance of shape constrained maximum likelihood estimators has been investigated in~\cite{palwoodroofe2007},
\cite{sereginwellner2010}, and~\cite{dosswellner2016}, whereas rates on Hellinger risk measures have been obtained in~\cite{sereginwellner2010},
\cite{kimsamworth2016}, and \cite{kimguntuboyinasamworth2016}.

In contrast with $L_p$-distances or the supremum distance, there is no distribution theory available for the Hellinger loss of shape constrained nonparametric estimators.
In this paper we present a first result in this direction, i.e., a central limit theorem for the Hellinger loss of Grenander type estimators for a monotone function~$\lambda$.
This type of isotonic estimators have also been considered by~\cite{durot2007}, and are defined as the left-hand slope of a concave (or convex) estimate of the primitive
of $\lambda$, based on $n$ observations.
We will establish our results under the same general setup of~\cite{durot2007},
which includes estimation of a probability density, a regression function, or a failure rate under monotonicity constraints.
In fact, after approximating the squared Hellinger distance by a weighted $L_2$-distance,
a central limit theorem can be obtained by mimicking the approach introduced in~\cite{durot2007}.
An interesting feature of our main result is that in the monotone density model, the variance of the limiting normal distribution
for the Hellinger distance does not depend on the underlying density.
This phenomena was also encountered for the $L_1$-distance in~\cite{groeneboom1985,groeneboom-hooghiemstra-lopuhaa1999}.

In Section~\ref{sec:definitions} we define the setup and approximate the squared Hellinger loss by a weighted $L_2$-distance.
A central limit theorem for the Hellinger distance is established in Section~\ref{sec:main result},
and we end the paper by a short discussion on the consequences for particular statistical models.

\section{Definitions and preparatory results}
\label{sec:definitions}
Consider the problem of estimating a non-increasing (or non-decreasing) function $\lambda:[0,1]\to\R$ on the basis of $n$ observations.
Suppose that we have at hand a cadlag step estimator $\Lambda_n$ for
\[
\Lambda(t)=\int_0^t\lambda(u)\,\mathrm{d}u,\qquad t\in[0,1].
\]
If $\lambda$ is non-increasing, then the Grenander-type estimator $\hat{\lambda}_n$ for $\lambda$ is defined as
the left-hand slope of the least concave majorant (LCM) of $\Lambda_n$, with $\hat\lambda_n(0)=\lim_{t\downarrow0}\hat\lambda_n(t)$.
If $\lambda$ is non-decreasing, then the Grenander-type estimator $\hat{\lambda}_n$ for $\lambda$ is defined as
the left-hand slope of the greatest convex minorant (GCM) of $\Lambda_n$, with $\hat\lambda_n(0)=\lim_{t\downarrow0}\hat\lambda_n(t)$.
We aim at proving the asymptotic normality of the Hellinger distance between $\hat{\lambda}_n$ and $\lambda$ defined by
\begin{equation}
\label{def:hellinger distance}
H(\hat{\lambda}_n,\lambda)
=
\left(
\frac{1}{2}\int_0^1
\left(
\sqrt{\hat{\lambda}_n(t)}-\sqrt{\lambda(t)}
\right)^2\,\mathrm{d}t
\right)^{1/2}.
\end{equation}
We will consider the same general setup as in~\cite{durot2007}, i.e.,
we will assume the following conditions
\begin{itemize}
\item[(A1)]
$\lambda$ is monotone and differentiable on $[0, 1]$ with $0<\inf_t|\lambda'(t)|\leq \sup_{t}|\lambda'(t)|<\infty$.
\item[(A2')]
Let $M_n=\Lambda_n-\Lambda$.
There exist $C>0$ such that for all $x>0$ and $t=0,1$,
\begin{equation}
\label{eq:bound expec increment}
\mathbb{E}
\left[
\sup_{u\in[0,1], x/2\leq|t-u|\leq x}
\left(
M_n(u)-M_n(t)
\right)^2
\right]
\leq
\frac{Cx}{n}.
\end{equation}
\end{itemize}
Durot~\cite{durot2007} also considered an additional condition (A2) in order to obtain bounds on $p$-th moments
(see Theorem~1 and Corollary~1 in~\cite{durot2007}).
However, we only need condition (A2') for our purposes.
\begin{itemize}
\item[(A3)]
$\hat\lambda_n(0)$ and $\hat\lambda_n(1)$ are stochastically bounded.
\item[(A4)]
Let $B_n$ be either a Brownian bridge or a Brownian motion.
There exists $q>12$, $C_q>0$, $L:[0,1]\mapsto \mathbb{R}$ and versions of $M_n=\Lambda_n-\Lambda$ and $B_n$, such that
\[
P
\left(
n^{1-1/q}
\sup_{t\in[0,1]}
\left|
M_n(t)-n^{-1/2}B_n\circ L(t)
\right|>x
\right)
\leq
C_qx^{-q}
\]
for $x\in(0,n]$.
Moreover, $L$ is increasing and twice differentiable on $[0,1]$ with $\sup_t|L''(t)|<\infty$ and $\inf_t L'(t)>0$.
\end{itemize}
In~\cite{durot2007} a variety of statistical models are discussed for which the above assumptions are satisfied,
such as estimation of a monotone probability density, a monotone regression function, and a monotone failure rate under right censoring.
In Section~\ref{sec:discussion}, we briefly discuss the consequence of our main result for these models.
We restrict ourselves to the case of a non-increasing function $\lambda$.
The case of non-decreasing $\lambda$ can be treated similarly.

The reason that one can expect a central limit theorem for the Hellinger distance is the fact that
the squared Hellinger distance can be approximated by a weighted squared $L_2$-distance.
This can be seen as follows,
\begin{equation}
\label{eq:approx by L2}
\begin{split}
\int_0^1
\left(\sqrt{\hat\lambda_n(t)}-\sqrt{\lambda(t)}\right)^2\,\mathrm{d}t
&=
\int_0^1
\left(\hat\lambda_n(t)-\lambda(t)\right)^2
\left(\sqrt{\hat\lambda_n(t)}+\sqrt{\lambda(t)}\right)^{-2}\,\mathrm{d}t\\
&\approx
\int_0^1
\left(\hat\lambda_n(t)-\lambda(t)\right)^2(4\lambda(t))^{-1}\,\mathrm{d}t.
\end{split}
\end{equation}
Since $L_2$-distances for Grenander-type estimators obey a central limit theorem
(e.g., see~\cite{groeneboom1985,groeneboom-hooghiemstra-lopuhaa1999,durot2002,kulikov-lopuhaa2005,durot2007}),
similar behavior might be expected for the squared Hellinger distance.
An application of the delta-method will then do the rest.

In order to make the approximation in~\eqref{eq:approx by L2} precise, we need the preparatory lemma below.
To this end, we introduce the inverse of~$\hat\lambda_n$, defined by
\begin{equation}
\label{def:inverse Uhat}
\hat U_n(a)
=
\argmax_{u\in[0,1]}
\left\{
\Lambda_n^+(u)-au
\right\},
\quad
\text{for all }a\in \mathbb{R},
\end{equation}
where
\[
\Lambda_n^+(t)
=
\max
\left\{
\Lambda_n(t),\lim_{u\uparrow t}\Lambda_n(u)
\right\}.
\]
Note that
\begin{equation}
\label{eq:switch relation}
\hat\lambda_n(t)> a
\Rightarrow
\hat U_n(a)\geq t.
\end{equation}
Furthermore, let $g$ denote the inverse of $\lambda$.
We then have the following result.
\begin{lemma}
\label{le:lambda}
Assume (A1), (A2'), (A3), and (A4).
Moreover, suppose that there are $C'>0$ and $s>3/4$ with
\begin{equation}
\label{eq:holder bound}
|\lambda'(t)-\lambda'(x)|
\leq
C'|t-x|^s,
\quad
\text{for all }t,x\in[0,1].
\end{equation}
Then
\[
\int_0^1|\hat{\lambda}_n(t)-\lambda(t)|^3\,\mathrm{d}t=o_P\left(n^{-5/6}\right).
\]
\end{lemma}
\begin{proof}
We follow the line of reasoning in the first step of the proof of Theorem 2 in~\cite{durot2007} with $p=3$.
For completeness we briefly sketch the main steps.
We will first show that
\[
\int_0^1|\hat{\lambda}_n(t)-\lambda(t)|^3\,\mathrm{d}t
=
\int_{\lambda(0)}^{\lambda(1)}|\hat{U}_n(b)-g(b)|^3 \lambda'(g(b))^2\,\mathrm{d}b
+
o_P(n^{-5/6}).
\]
To this end, consider
\[
I_1
=
\int_0^1
\left(\hat\lambda_n(t)-\lambda(t)\right)_+^3
\,\mathrm{d}t,
\qquad
I_2
=
\int_0^1
\left(\lambda(t)-\hat\lambda_n(t)\right)_+^3
\,\mathrm{d}t,
\]
where $x_+=\max\{x,0\}$.
We approximate $I_1$ by
\[
J_1
=
\int_0^1
\int_0^{(\lambda(0)-\lambda(t))^3}
\mathds{1}_{\left\{\hat\lambda_n(t)\geq \lambda(t)+a^{1/3}\right\}}
\,\mathrm{d}a\,\mathrm{d}t.
\]
From the reasoning on page 1092 of~\cite{durot2007}, we deduce that
\[
0
\leq
I_1-J_1
\leq
\int_0^{n^{-1/3}\log n}
\left(\hat\lambda_n(t)-\lambda(t)\right)_+^3
\,\mathrm{d}t
+
|\hat\lambda_n(0)-\lambda(1)|^3
\mathds{1}_{\left\{n^{1/3}\hat U_n(\lambda(0))>\log n\right\}}.
\]
Since the $\hat\lambda_n(0)$ is stochastically bounded and $\lambda(1)$ is bounded,
together with Lemma~4 in~\cite{durot2007}, the second term is of the order $o_p(n^{-5/6})$.
Furthermore, for the first term we can choose $p'\in[1,2)$ such that the first term on the right hand side is bounded by
\[
|\hat\lambda_n(0)-\lambda(1)|^{3-p'}
\int_0^{n^{-1/3}\log n}
|\hat\lambda_n(t)-\lambda(t)|^{p'}\,\mathrm{d}t.
\]
As in~\cite{durot2007}, we get
\[
\mathbb{E}\left[
\int_0^{n^{-1/3}\log n}
|\hat\lambda_n(t)-\lambda(t)|^{p'}
\,\mathrm{d}t
\right]
\leq
K
n^{-(1+p')/3}\log n
=
o(n^{-5/6}),
\]
by choosing $p'\in(3/2,2)$.
It follows that $I_1=J_1+o_P(n^{-5/6})$.
By a change of variable $b=\lambda(t)+a^{1/3}$, we find
\[
I_1
=
\int_{\lambda(1)}^{\lambda(0)}
\int_{g(b)}^{\hat U_n(b)}
3(b-\lambda(t))^2
\mathds{1}_{\left\{g(b)<\hat U_n(b)\right\}}\,\mathrm{d}t\,\mathrm{d}b
+
o_p(n^{-5/6}).
\]
Then, by a Taylor expansion, (A1) and~\eqref{eq:holder bound}, there exists a $K>0$, such that
\begin{equation}
\label{eq:taylor}
\left|
\left(b-\lambda(t)\right)^2
-
\big\{
\left(
g(b)-t
\right)
\lambda'(g(b))
\big\}^2
\right|
\leq
K
\left(
t-g(b)
\right)^{2+s},
\end{equation}
for all $b\in(\lambda(1),\lambda(0))$ and $t\in(g(b),1]$.
We find
\begin{equation}
\label{eq:approx I1}
I_1
=
\int_{\lambda(1)}^{\lambda(0)}
\int_{g(b)}^{\hat U_n(b)}
3(t-g(b))^2\lambda'(g(b))^2
\mathds{1}_{\left\{g(b)<\hat U_n(b)\right\}}\,\mathrm{d}t\,\mathrm{d}b
+
R_n
+
o_p(n^{-5/6}),
\end{equation}
where
\[
\begin{split}
|R_n|
&\leq
\int_{\lambda(1)}^{\lambda(0)}
\int_{g(b)}^{\hat U_n(b)}
3K(t-g(b))^{2+s}
\mathds{1}_{\left\{g(b)<\hat U_n(b)\right\}}\,\mathrm{d}t\,\mathrm{d}b\\
&=
\frac{3K}{3+s}
\int_{\lambda(1)}^{\lambda(0)}
|\hat U_n(b)-g(b)|^{3+s}
\,\mathrm{d}b
=
O_p(n^{-(3+s)/3})
=
o_p(n^{-5/6}),
\end{split}
\]
by using (23) from~\cite{durot2007}, i.e.,
for every $q'<3(q-1)$, there exists $K_{q'}>0$ such that
\begin{equation}
\label{eq:bound Un}
\mathbb{E}
\left[
\left(n^{1/3}|\hat U_n(a)-g(a)|\right)^{q'}
\right]
\leq
K_{q'},
\quad
\text{for all }a\in \mathbb{R}.
\end{equation}
It follows that
\[
I_1
=
\int_{\lambda(1)}^{\lambda(0)}
\left(
\hat U_n(b)-g(b)
\right)^3
\lambda'(g(b))^2
\mathds{1}_{\left\{g(b)<\hat U_n(b)\right\}}\,\mathrm{d}b
+
o_p(n^{-5/6}).
\]
In the same way, one finds
\[
I_2=
\int_{\lambda(1)}^{\lambda(0)}
\left(
g(b)-\hat U_n(b)
\right)^{3}
\lambda'(g(b))^2
\mathds{1}_{\left\{g(b)>\hat U_n(b)\right\}}\,\mathrm{d}b
+
o_p(n^{-5/6}),
\]
and it follows that
\[
\int_0^1
|\hat{\lambda}_n(t)-\lambda(t)|^3\,\mathrm{d}t
=
I_1+I_2
=
\int_{\lambda(1)}^{\lambda(0)}
|\hat U_n(b)-g(b)|^3
\lambda'(g(b))^2
\,\mathrm{d}b
+
o_p(n^{-5/6}).
\]
Now, since $\lambda'$ is bounded, by Markov's inequality, for each $\epsilon>0$,  we can write
\[
\begin{split}
&
\p\left(
n^{5/6}\int_{\lambda(0)}^{\lambda(1)}
|\hat{U}_n(b)-g(b)|^3 \lambda'(g(b))^2
\,\mathrm{d}b
>\epsilon
\right)\\
&\quad\leq
\frac{1}{c\epsilon n^{1/6}}
\int_{\lambda(0)}^{\lambda(1)}
\E\left[n|\hat{U}_n(b)-g(b)|^3 \right]\,\mathrm{d}b
\leq
Kn^{-1/6}\to 0.
\end{split}
\]
For the last inequality we again used~\eqref{eq:bound Un} with $q'=3$.
It follows that
\begin{equation}
\label{eqn:U_n}
\int_{\lambda(0)}^{\lambda(1)}|\hat{U}_n(b)-g(b)|^3 \lambda'(g(b))^2
\,\mathrm{d}b=o_P(n^{5/6}),
\end{equation}
which finishes the proof.
\end{proof}
The approximation in~\eqref{eq:approx by L2} can now be made precise.
\begin{lemma}
\label{lem:approx by L2}
Under the conditions of Lemma~\ref{le:lambda} and if $\lambda$ is strictly positive,
we have that
\[
\int_0^1
\left(\sqrt{\hat\lambda_n(t)}-\sqrt{\lambda(t)}\right)^2\,\mathrm{d}t
=
\int_0^1
\left(\hat\lambda_n(t)-\lambda(t)\right)^2(4\lambda(t))^{-1}\,\mathrm{d}t
+
o_p(n^{-5/6}).
\]
\end{lemma}
\begin{proof}
Similar to~\eqref{eq:approx by L2}, we write
\[
\int_0^1
\left(
\sqrt{\hat{\lambda}_n(t)}-\sqrt{\lambda(t)}
\right)^2\,\mathrm{d}t
=
\int_0^1
\left(\hat\lambda_n(t)-\lambda(t)\right)^2(4\lambda(t))^{-1}\,\mathrm{d}t
+R_n,
\]
where
\[
R_n
=
\int_0^1
\left(\hat{\lambda}_n(t)-\lambda(t)\right)^2
\left\{
\left(
\sqrt{\hat{\lambda}_n(t)}+\sqrt{\lambda(t)}
\right)^{-2}
-
(4\lambda(t))^{-1}
\right\}\,\mathrm{d}t.
\]
Write
\[
\begin{split}
4\lambda(t)-\left(\sqrt{\hat{\lambda}_n(t)}+\sqrt{\lambda(t)}\right)^2
&=
\lambda(t)-\hat{\lambda}_n(t)-2\sqrt{\lambda(t)}\left(\sqrt{\hat{\lambda}_n(t)}-\sqrt{\lambda(t)}\right)\\
&=
\left(\lambda(t)-\hat{\lambda}_n(t)\right)
\left(
1+\frac{2\sqrt{\lambda(t)}}{\sqrt{\hat{\lambda}_n(t)}+\sqrt{\lambda(t)}}
\right).
\end{split}
\]
Since $0<\lambda(1)\leq \lambda(t)\leq\lambda(0)<\infty$, this implies that
\[
|R_n|
\leq
\int_0^1
\left(\hat{\lambda}_n(t)-\lambda(t)\right)^2
\frac{\left|4\lambda(t)-\left(\sqrt{\hat{\lambda}_n(t)}+\sqrt{\lambda(t)}\right)^2\right|}{4\lambda(t)\left(\sqrt{\hat{\lambda}_n(t)}+\sqrt{\lambda(t)} \right)^2}
\,\mathrm{d}t
\leq
C\int_0^1
\left|\hat{\lambda}_n(t)-\lambda(t)\right|^3\,\mathrm{d}t,
\]
for some positive constant $C$ only depending on $\lambda(0)$ and $\lambda(1)$.
Then, from Lemma~\ref{le:lambda}, it follows that  $n^{5/6}R_n=o_P(1)$.
\end{proof}

\section{Main result}
\label{sec:main result}
In order to formulate the central limit theorem for the Hellinger distance,
we introduce the process $X$, defined as
\begin{equation}
\label{def:X}
X(a)=\argmax_{u\in\R}
\left\{W(u)-(u-a)^2\right\},\qquad a\in\R,
\end{equation}
with $W$ being a standard two-sided Brownian motion.
This process was introduced and investigated in~\cite{groeneboom1985,groeneboom1989} and plays a key role in
the asymptotic behavior of isotonic estimators.
The distribution of the random variable $X(0)$ is the pointwise limiting distribution of several isotonic estimators
and the constant
\begin{equation}
\label{def:k2}
k_2
=
\int_0^\infty \mathrm{cov}\left(|X(0)|^2,|X(a)-a|^2\right)\,\mathrm{d}a,
\end{equation}
appears in the limit variance of the $L_p$-error of isotonic estimators
(e.g., see~\cite{groeneboom1985}, \cite{groeneboom-hooghiemstra-lopuhaa1999}, \cite{durot2002},
\cite{kulikov-lopuhaa2005}, and~\cite{durot2007}).
We then have the following central limit theorem for the squared Hellinger loss.
\begin{theo}
\label{theo:hellinger}
Assume (A1), (A2'), (A3), (A4), and~\eqref{eq:holder bound}.
Moreover, suppose that $\lambda$ is strictly positive.
Then, the following holds
\[
n^{1/6}
\left\{
n^{2/3}\int_0^1
\left(
\sqrt{\hat{\lambda}_n(t)}-\sqrt{\lambda(t)}
\right)^2\,\mathrm{d}t-\mu^2
\right\}
\to N(0,\sigma^2),
\]
where
\[
\mu^2=\E\left[|X(0)|^2\right] \int_0^1\frac{|\lambda'(t)\,L'(t)|^{2/3}}{2^{2/3}\lambda(t)}\,\mathrm{d}t,
\qquad
\sigma^2=2^{1/3}k_2\int_0^1\frac{|\lambda'(t)\,L'(t)|^{2/3}L'(t)}{\lambda(t)^2}\,\mathrm{d}t,
\]
where $k_2$ is defined in~\eqref{def:k2}.
\end{theo}
\begin{proof}
According to Lemma~\ref{lem:approx by L2},
it is sufficient to show that $n^{1/6}\left(n^{2/3} I_n-\mu^2\right)\to N(0,\sigma^2)$, with
\[
I_n
=
\int_0^1
\left(\hat{\lambda}_n(t)-\lambda(t)\right)^2
(4\lambda(t))^{-1}\,\mathrm{d}t.
\]
Again, we follow the same line of reasoning as in the proof of Theorem~2 in~\cite{durot2007}.
We briefly sketch the main steps of the proof.
We first express $I_n$ in terms of the inverse process $\hat U_n$, defined in~\eqref{def:inverse Uhat}.
To this end, similar to the proof of Lemma~\ref{le:lambda}, consider
\[
\tilde{I}_1
=
\int_0^1
\left(\hat{\lambda}_n(t)-\lambda(t)\right)_+^2(4\lambda(t))^{-1}\,\mathrm{d}t,
\qquad
\tilde{I}_2
=
\int_0^1
\left(\lambda(t)-\hat{\lambda}_n(t)\right)_+^2(4\lambda(t))^{-1}\,\mathrm{d}t.
\]
For the first integral, we can now write
\[
\tilde{I}_1
=
\int_0^1\int_0^\infty
\1_{\left\{\hat{\lambda}_n(t)\geq\lambda(t)+\sqrt{4a\lambda(t)}\right\}}\,\mathrm{d}a\,\mathrm{d}t.
\]
Then, if we introduce
\begin{equation}
\label{def:J1}
\tilde{J}_1=
\int_0^1 \int_0^{(\lambda(0)-\lambda(t))^2/4\lambda(t)}
\1_{\left\{\hat{\lambda}_n(t)\geq\lambda(t)+\sqrt{4a\lambda(t)}\right\}}\,\mathrm{d}a\,\mathrm{d}t,
\end{equation}
we obtain
\[
\begin{split}
0
\leq
\tilde{I}_1-\tilde{J}_1
&\leq
\int_{0}^{\hat U_n(\lambda(0))}
\int_{(\lambda(0)-\lambda(t))^2/4a\lambda(t)}^\infty
\1_{\left\{\hat{\lambda}_n(t)\geq\lambda(t)+\sqrt{4a\lambda(t)}\right\}}\,\mathrm{d}a\,\mathrm{d}t\\
&\leq
\frac{1}{4\lambda(1)}
\int_0^{\hat{U}_n(\lambda(0))}
\left(\hat{\lambda}_n(t)-\lambda(t)\right)_+^2\,\mathrm{d}t.
\end{split}
\]
Similar to the reasoning in the proof of Lemma~\ref{le:lambda}, we conclude that $\tilde{I}_1=\tilde{J}_1+o_p(n^{-5/6})$.
Next, the change of variable $b=\lambda(t)+\sqrt{4a\lambda(t)}$ yields
\begin{equation}
\label{eqn:eq1}
\begin{split}
\tilde{J}_1
&=
\int_{\lambda(1)}^{\lambda(0)}\int_{g(b)}^{\hat{U}_n(b)}
\frac{b-\lambda(t)}{2\lambda(t)}
\1_{\left\{\hat{U}_n(b)>g(b)\right\}}
\,\mathrm{d}t\,\mathrm{d}b\\
&=
\int_{\lambda(1)}^{\lambda(0)}\int_{g(b)}^{\hat{U}_n(b)}
\frac{b-\lambda(t)}{2b}
\1_{\left\{\hat{U}_n(b)>g(b)\right\}}
\,\mathrm{d}t\,\mathrm{d}b\\
&\qquad+
\int_{\lambda(1)}^{\lambda(0)}\int_{g(b)}^{\hat{U}_n(b)}
\frac{\left(b-\lambda(t)\right)^2}{2b\lambda(t)}
\1_{\left\{\hat{U}_n(b)>g(b)\right\}}
\,\mathrm{d}t\,\mathrm{d}b.
\end{split}
\end{equation}
Let us first consider the second integral on the right hand side of~\eqref{eqn:eq1}.
We then have
\[
\begin{split}
&
\int_{\lambda(1)}^{\lambda(0)}\int_{g(b)}^{\hat{U}_n(b)}
\frac{\left(b-\lambda(t)\right)^2}{2b\lambda(t)}
\1_{\left\{\hat{U}_n(b)>g(b)\right\}}
\,\mathrm{d}t\,\mathrm{d}b\\
&\quad\leq
\frac{1}{2\lambda(1)^2}
\int_{\lambda(1)}^{\lambda(0)}\int_{g(b)}^{\hat{U}_n(b)}
\left(b-\lambda(t)\right)^2
\1_{\left\{\hat{U}_n(b)>g(b)\right\}}
\,\mathrm{d}t\,\mathrm{d}b\\
&\quad\leq
\frac{1}{2\lambda(1)^2}
\sup_{x\in[0,1]}|\lambda'(x)|
\int_{\lambda(1)}^{\lambda(0)}
\1_{\left\{\hat{U}_n(b)>g(b)\right\}}
\int_{g(b)}^{\hat{U}_n(b)}
\left(t-g(b)\right)^2\,\mathrm{d}t\,\mathrm{d}b\\
&\quad=
\frac{1}{6\lambda(1)^2}
\sup_{x\in[0,1]}|\lambda'(x)|
\int_{\lambda(1)}^{\lambda(0)}
\1_{\left\{\hat{U}_n(b)>g(b)\right\}}
\left(\hat{U}_n(b)-g(b)\right)^3\,\mathrm{d}b
=
o_P(n^{-5/6}),
\end{split}
\]
again by using~\eqref{eq:bound Un} with $q'=3$.
Then consider the first integral on the right hand side of~\eqref{eqn:eq1}.
Similar to~\eqref{eq:taylor}, there exists $K>0$ such that
\[
\left|
(b-\lambda(t)-(g(b)-t)\lambda'(g(b)))
\right|
\leq
K(t-g(b))^{1+s},
\]
for all $b\in(\lambda(1),\lambda(0))$ and $t\in(g(b),1]$.
Taking into account that $\lambda'(g(b))<0$, similar to~\eqref{eq:approx I1}, it follows that
\[
\begin{split}
\tilde I_1
&=
\int_{\lambda(1)}^{\lambda(0)}\int_{g(b)}^{\hat{U}_n(b)}
\frac{|\lambda'(g(b))|}{2b}\left(t-g(b)\right)
\1_{\left\{\hat{U}_n(b)>g(b)\right\}}
\,\mathrm{d}t\,\mathrm{d}b
+
\tilde R_n
+
o_p(n^{-5/6}),
\end{split}
\]
where
\[
\begin{split}
|\tilde R_n|
&\leq
\int_{\lambda(1)}^{\lambda(0)}
\int_{g(b)}^{\hat U_n(b)}
2K(t-g(b))^{1+s}
\mathds{1}_{\left\{g(b)<\hat U_n(b)\right\}}\,\mathrm{d}t\,\mathrm{d}b\\
&=
\frac{2K}{2+s}
\int_{\lambda(1)}^{\lambda(0)}
|\hat U_n(b)-g(b)|^{2+s}
\,\mathrm{d}b
=
O_p(n^{-(2+s)/3})
=
o_p(n^{-5/6}),
\end{split}
\]
by using~\eqref{eq:bound Un} once more, and the fact that $s>3/4$.
It follows that
\[
\tilde I_1
=
\int_{\lambda(1)}^{\lambda(0)}
\frac{|\lambda'(g(b))|}{4b}\left(\hat{U}_n(b)-g(b)\right)^2
\1_{\left\{\hat{U}_n(b)>g(b)\right\}}
\,\mathrm{d}b
+
o_p(n^{-5/6}).
\]
In the same way
\[
\tilde I_2
=
\int_{\lambda(1)}^{\lambda(0)}
\frac{|\lambda'(g(b))|}{4b}\left(\hat{U}_n(b)-g(b)\right)^2
\1_{\left\{\hat{U}_n(b)<g(b)\right\}}
\,\mathrm{d}b
+
o_p(n^{-5/6}),
\]
so that
\[
I_n=\tilde I_1+\tilde I_2
=
\int_{\lambda(1)}^{\lambda(0)}
\left(\hat{U}_n(b)-g(b)\right)^2
\frac{|\lambda'(g(b))|}{4b}
\,\mathrm{d}b+o_P(n^{-5/6}).
\]

We then mimic step 2 in the proof of Theorem~2 in~\cite{durot2007}.
Consider the representation
\[
B_n(t)=W_n(t)-\xi_n t,
\]
where $W_n$  is a standard Brownian motion, $\xi_n=0$ if $B_n$ is Brownian motion,
and $\xi_n$ is a standard normal random variable independent of $B_n$,
if $B_n$ is a Brownian bridge.
Then, define
\[
\mathbb{W}_{t}(u)
=
n^{1/6}
\left\{
W_n(L(t)+n^{-1/3})-W_n(L(t))
\right\},
\quad
\text{for }t\in[0,1],
\]
which has the same distribution as a standard Brownian motion.
Now, for $t\in[0,1]$, let $d(t)=|\lambda'(t)|/(2L'(t)^2)$ and define
\begin{equation}
\label{def:tilde V}
\tilde V(t)
=
\argmax_{|u|\leq\log n}
\left\{
\mathbb{W}_{t}(u)-d(t)u^2
\right\}.
\end{equation}
Then similar to~(26) in~\cite{durot2007}, we will obtain
\begin{equation}
\label{eqn:eq2}
n^{2/3}I_n
=
\int_0^1\left|\tilde{V}(t)-n^{-1/6}\frac{\xi_n}{2d(t)} \right|^2\left|\frac{\lambda'(t)}{L'(t)}\right|^2\frac{1}{4\lambda(t)}\,\mathrm{d}t+o_P(n^{-1/6}).
\end{equation}
To prove~\eqref{eqn:eq2}, by using the approximation
\[
\hat U_n(a)-g(a)\approx\frac{L(\hat U_n(a))-L(g(a))}{L'(g(a))}
\]
and a change of variable $a^{\xi}=a-n^{1/2}\xi_nL'(g(a))$, we first obtain
\[
n^{2/3}I_n
=
n^{2/3}
\int_{\lambda(1)+\delta_n}^{\lambda(0)-\delta_n}
\left|
L(\hat U_n(a^{\xi}))-L(g(a^{\xi}))
\right|^2
\frac{|\lambda'(g(a))|}{(L'(g(a)))^2}\frac{1}{4a}
\,\mathrm{d}a+o_p(n^{-1/6}),
\]
where $\delta_n=n^{-1/6}/\log n$.
Apart from the factor $1/4a$,
the integral on the right hand side is the same as in the proof of Theorem~2 in~\cite{durot2007} for $p=2$.
This means that we can apply the same series of succeeding approximations for
$L(\hat U_n(a^{\xi}))-L(g(a^{\xi}))$ as in~\cite{durot2007}, which yields
\[
n^{2/3}I_n
=
n^{2/3}
\int_{\lambda(1)+\delta_n}^{\lambda(0)-\delta_n}
\left|
\tilde V(g(a))-n^{-1/6}\frac{\xi_n}{2d(g(a))}
\right|^2
\frac{|\lambda'(g(a))|}{(L'(g(a)))^2}\frac{1}{4a}
\,\mathrm{d}a+o_p(n^{-1/6}).
\]
Finally, because the integrals over $[\lambda(1),\lambda(1)+\delta_n]$
and $[\lambda(0)-\delta_n,\lambda(0)]$ are of the order~$o_p(n^{-1/6})$,
this yields~\eqref{eqn:eq2} by a change of variables $t=g(a)$.

The next step is to show that the term with $\xi_n$ can be removed from~\eqref{eqn:eq2}.
This can be done exactly as in~\cite{durot2007}, since the only difference with the corresponding integral in~\cite{durot2007}
is the factor $1/4\lambda(t)$, which is bounded and does not influence the argument in~\cite{durot2007}.
We find that
\[
n^{2/3}I_n
=
\int_0^1
|\tilde{V}(t)|^2\left|\frac{\lambda'(t)}{L'(t)}\right|^2\frac{1}{4\lambda(t)}\,\mathrm{d}t+o_P(n^{-1/6}).
\]

Then define
\begin{equation}
\label{def:Yn}
Y_n(t)=\left(|\tilde{V}(t)|^2-\E\left[|\tilde{V}(t)|^2\right]\right)\left|\frac{\lambda'(t)}{L'(t)}\right|^2\frac{1}{4\lambda(t)}.
\end{equation}
By approximating $\tilde V(t)$ by
\[
V(t)
=
\argmax_{u\in\R}
\left\{
\mathbb{W}_{t}(u)-d(t)u^2
\right\},
\]
and using that, by Brownian scaling, $d(t)^{2/3}V(t)$ has the same distribution as $X(0)$,
see~\cite{durot2007} for details, we have that
\[
\begin{split}
\int_0^1
\E\left[|\tilde{V}(t)|^2\right]
\left|\frac{\lambda'(t)}{L'(t)}\right|^2\frac{1}{4\lambda(t)}\,\mathrm{d}t
&=
\E\left[|X(0)|^2\right]
\int_0^1d(t)^{-4/3} \left|\frac{\lambda'(t)}{L'(t)}\right|^2\frac{1}{4\lambda(t)}\,\mathrm{d}t
+
o(n^{-1/6})\\
&=
\mu^2+o(n^{-1/6}).
\end{split}
\]
It follows that
\[
n^{1/6}(I_n-\mu^2)=n^{1/6}\int_0^1Y_n(t)\,\mathrm{d}t+o_P(1).
\]
We then first show that
\begin{equation}
\label{eq:limit variance}
\mathrm{Var}\left(n^{1/6}\int_0^1Y_n(t)\,\mathrm{d}t\right)\to\sigma^2.
\end{equation}
Once more, following the proof in~\cite{durot2007} we have
\[
\begin{split}
v_n
&=
\mathrm{Var}\left(\int_0^1Y_n(t)\,\mathrm{d}t\right)\\
&=
2\int_0^1\int_s^1\left|\frac{\lambda'(t)}{L'(t)}\frac{\lambda'(s)}{L'(s)} \right|^2\frac{1}{4\lambda(t)}\frac{1}{4\lambda(s)}
\mathrm{cov}(|\tilde{V}(t)|^2,|\tilde{V}(s)|^2)\,\mathrm{d}t\,\mathrm{d}s.
\end{split}
\]
After the same sort of approximations as in~\cite{durot2007}, we get
\[
v_n
=
2\int_0^1\int_s^{\min(1,s+c_n)}
\left|\frac{\lambda'(s)}{L'(s)}\right|^4\frac{1}{(4\lambda(s))^2}
\mathrm{cov}(|V_t(s)|^2,|V_s(s)|^2)\,\mathrm{d}t\,\mathrm{d}s
+
o(n^{-1/3}),
\]
where $c_n=2n^{-1/3}\log n/\inf_t L'(t)$ and where, for all $s$ and $t$,
\[
V_t(s)
=
\argmax_{u\in\R}
\left\{
\mathbb{W}_t(u)-d(s)u^2
\right\}.
\]
Then use that $d(s)^{2/3}V_t(s)$ has the same distribution as
\[
X\big(n^{1/3}d(s)^{2/3}\big(L(t)-L(s)\big)\big)-n^{1/3}d(s)\big(L(t)-L(s)\big),
\]
so that the change of variable $a=n^{1/3}d(s)^{2/3}(L(t)-L(s))$ in $v_n$ leads to
\[
\begin{split}
n^{1/3}v_n
&\to
2\int_0^1\int_{0}^{\infty}
\left|\frac{\lambda'(s)}{L'(s)}\right|^4\frac{1}{(4\lambda(s))^2}
\frac{1}{d(s)^{10/3}L'(s)}\mathrm{cov}(|X(a)|^2,|X(0)|^2)\,\mathrm{d}a\,\mathrm{d}s\\
&\to
2k_2\int_0^1
\left|\frac{\lambda'(s)}{L'(s)}\right|^4\frac{1}{(4\lambda(s))^2}
\frac{2^{10/3}|L'(s)|^{17/3}}{|\lambda'(s)|^{10/3}}
\,\mathrm{d}s
=
\sigma^2,
\end{split}
\]
which proves~\eqref{eq:limit variance}.

Finally, asymptotic normality of $n^{1/6}\int_0^1Y_n(t)\,\mathrm{d}t$ follows by Bernstein's method of big blocks and small blocks
in the same way as in step~6 of the proof of Theorem~2 in~\cite{durot2007}.
\end{proof}

\begin{cor}
\label{cor:hellinger}
Assume (A1), (A2'), (A3), (A4), and~\eqref{eq:holder bound} and let $H(\hat{\lambda}_n,\lambda)$
be the Hellinger distance defined in~\eqref{def:hellinger distance}.
Moreover, suppose that $\lambda$ is strictly positive.
Then,
\[
n^{1/6}
\left\{
n^{1/3}H(\hat{\lambda}_n,\lambda)-\tilde{\mu}
\right\}
\to N(0,\tilde{\sigma}^2),
\]
where
$\tilde{\mu}=2^{-1/2}\mu$ and $\tilde{\sigma}^2=\sigma^2/8\mu^2$, where $\mu^2$ and $\sigma^2$ are defined in
Theorem~\ref{theo:hellinger}.
\end{cor}
\begin{proof}
This follows immediately by applying the delta method with $\phi(x)=2^{-1/2}\sqrt{x}$ to the result in Theorem~\ref{theo:hellinger}.
\end{proof}

\section{Discussion}
\label{sec:discussion}
The type of scaling for the Hellinger distance in Corollary~\ref{cor:hellinger} is similar to that
in the central limit theorem for $L_p$-distances.
This could be expected in view of the approximation in terms of a weighted squared $L_2$-distance, see Lemma~\ref{lem:approx by L2},
and the results, e.g., in~\cite{kulikov-lopuhaa2005} and~\cite{durot2007}.
Actually, this is not always the case.  
The phenomenon of observing different speeds of convergence for the Hellinger distance from those we for the $L_1$ and $L_2$ norms is considered in~\cite{Birge86}. 
In fact, this is related to the existence of a lower bound for the function we are estimating. 
If the function of interest is bounded from below, which is the case considered in this paper, then the approximation~\eqref{eq:approx by L2} holds, see~\cite{Birge86} for an explanation.

When we insert the expressions for $\mu^2$ and $\sigma^2$ from Theorem~\ref{theo:hellinger}, then we get
\[
\tilde\sigma^2
=
\frac{k_2}{4\E\left[|X(0)|^2\right] }
\frac{\int_0^1|\lambda'(t)L'(t)|^{2/3}L'(t)\lambda(t)^{-2}\,\mathrm{d}t}{
\int_0^1|\lambda'(t)L'(t)|^{2/3}\lambda(t)^{-1}\,\mathrm{d}t},
\]
where $k_2$ is defined in~\eqref{def:k2}.
This means that in statistical models where $L=\Lambda$ in condition~(A4), and hence~$L'=\lambda$, the limiting variance $\tilde\sigma^2=k_2/(4\E[|X(0)|^2])$
does not depend on $\lambda$.

One such a model is estimation of the common monotone density $\lambda$ on $[0,1]$ of independent random variables $X_1,\ldots,X_n$.
Then, $\Lambda_n$ is the empirical distribution function of $X_1,\ldots,X_n$ and $\hat\lambda_n$ is Grenander's estimator~\cite{grenander1956}.
In that case, if $\inf_t\lambda(t)>0$, the conditions of Corollary~\ref{cor:hellinger} are satisfied with $L=\Lambda$
(see Theorem~6 in~\cite{durot2007}),
so that the limiting variance of the Hellinger loss for the Grenander estimator does not depend on the underlying density.
This behavior was conjectured in~\cite{wellner2015} and coincides with that of the limiting variance in the central
limit theorem for the $L_1$-error for the Grenander estimator, first discovered by~\cite{groeneboom1985}
(see also~\cite{groeneboom-hooghiemstra-lopuhaa1999,durot2002} and~\cite{kulikov-lopuhaa2005,durot2007}).

Another example is when we observe independent identically distributed inhomogeneous Poisson processes $N_1,\ldots,N_n$
with common mean function $\Lambda$ on $[0,1]$ with derivative $\lambda$, for which $\Lambda(1)<\infty$.
Then $\Lambda_n$ is the restriction of $(N_1+\cdots+N_n)/n$ to $[0,1]$.
Also in that case, the conditions of Corollary~\ref{cor:hellinger} are satisfied with $L=\Lambda$
(see Theorem~4 in~\cite{durot2007}),
so that the limiting variance of the Hellinger loss for $\hat\lambda_n$ does not depend on the common underlying intensity $\lambda$.
However, note that for this model, the $L_1$-loss for $\hat\lambda_n$ is asymptotically normal according to Theorem~2 in~\cite{durot2007},
but with limiting variance depending on the value $\Lambda(1)-\Lambda(0)$.

Consider the monotone regression model $y_{i,n}=\lambda(i/n)+\epsilon_{i,n}$, for $i=1,\ldots,n$,
where the~$\epsilon_{i,n}$'s are i.i.d.~random variables with mean zero and variance $\sigma^2>0$.
Let
\[
\Lambda_n(t)=\frac1n\sum_{i\leq nt} y_{i,n},
\quad
t\in[0,1],
\]
be the empirical distribution function.
Then $\hat\lambda_n$ is (a slight modification of) Brunk's estimator from~\cite{brunk1958}.
Under appropriate moment conditions on the $\epsilon_{i,n}$,
the conditions of Corollary~\ref{cor:hellinger} are satisfied with $L(t)=t\sigma^2$
(see Theorem~5 in~\cite{durot2007}).
In this case, the limiting variance of the Hellinger loss for $\hat\lambda_n$ depends on both $\lambda$ and $\sigma^2$,
whereas the  the $L_1$-loss for~$\hat\lambda_n$ is asymptotically normal according to Theorem~2 in~\cite{durot2007},
but with limiting variance only depending on $\sigma^2$.

Suppose we observe a right-censored sample $(X_1,\Delta_1),\ldots,(X_n,\Delta_n)$, where
$X_i=\min(T_i,Y_i)$ and $\Delta_i=\1_{\{T_i\leq Y_i\}}$, with
the $T_i$'s being nonnegative i.i.d.~failure times and the $Y_i$'s are i.i.d.~censoring times independent of the $T_i$'s.
Let $F$ be the distribution function of the $T_i$'s with density $f$ and let $G$ be the distribution function of the $Y_i$'s.
The parameter of interest is the monotone failure rat $\lambda=f/(1-F)$ on $[0,1]$.
In this case, $\Lambda_n$ is the restriction of he Nelson-Aalen estimator to $[0,1]$.
If we assume (A1) and $\inf_t \lambda(t)>0$, then under suitable assumptions on $F$ and $G$ the conditions of
Corollary~\ref{cor:hellinger} hold with
\[
L(t)=\int_0^t \frac{\lambda(u)}{(1-F(u)))(1-G(u))}\,\mathrm{d}u,
\quad
t\in[0,1],
\]
(see Theorem~3 in~\cite{durot2007}).
This means that the limiting variance of the Hellinger loss depends on $\lambda$, $F$ and $G$, whereas
the  limiting variance of the $L_1$-loss depends only on their values at~0 and~1.
In particular, in the case of nonrandom censoring times, $L=(1-F)^{-1}-1$, the limiting variance of the Hellinger loss
depends on $\lambda$ and $F$,
whereas the  limiting variance of the $L_1$-loss depends only on the value $F(1)$.

\printbibliography
\end{document}